\documentclass{amsart}
\usepackage{enumitem,amsmath, amssymb}
\usepackage{graphicx,bm}
\usepackage{cite}
\usepackage[usenames]{color}
\usepackage[colorlinks=true,linkcolor=webgreen,filecolor=webbrown,citecolor=webgreen]{hyperref}%
\definecolor{webgreen}{rgb}{0,0,0.6}
\definecolor{webbrown}{rgb}{.6,0,0}

\newtheorem{theorem}{Theorem}[section]

\newtheorem{conjecture}[theorem]{Conjecture}
\newtheorem{corollary}[theorem]{Corollary}
\newtheorem{definition}[theorem]{Definition}
\newtheorem{example}[theorem]{Example}
\newtheorem{lemma}[theorem]{Lemma}

\newtheorem{proposition}[theorem]{Proposition}

\newtheorem{remark}[theorem]{Remark}
\newtheorem{question}{Question}

\begin{document}

\title{Sparse subsets of the natural numbers and Euler's totient function}

\author{Mithun Kumar Das\textsuperscript{1}, Pramod Eyyunni\textsuperscript{2} \and Bhuwanesh Rao Patil\textsuperscript{3} }

\address{Harish-Chandra Research Institute, HBNI, Chhatnag Road, Jhunsi, Allahabad - 211019, Uttar Pradesh, India.\\
Email: \textsuperscript{1}das.mithun3@gmail.com, \textsuperscript{2}pramodeyy@gmail.com \and\\ \textsuperscript{3}bhuwanesh1989@gmail.com }

\maketitle


\begin{abstract}
In this article, we investigate sparse subsets of the natural numbers and study the sparseness of some sets 
associated to the Euler's totient function $\phi$ via the property of `Banach Density'.  
These sets related to the totient function are defined as follows: $V:=\phi(\mathbb{N})$ and $N_i:=\{N_i(m)\colon m\in V \}$ for $i = 1, 2, 3,$ 
where $N_1(m)=\max\{x\in \mathbb{N}\colon \phi(x)\leq m\}$,
$N_2(m)=\max(\phi^{-1}(m))$ and  $N_3(m)=\min(\phi^{-1}(m))$ for $ m\in V$.  Masser and Shiu call the elements of $N_1$ as 
`sparsely totient numbers' and construct an infinite family of these numbers. Here we construct several infinite families
of numbers in $N_2\setminus N_1$ and an infinite family of composite numbers in $N_3$. 
We also study (i) the ratio
$\frac{N_2(m)}{N_3(m)}$, which is linked to the Carmichael's conjecture, namely, $|\phi^{-1}(m)|\geq 2 ~\forall ~ m\in V$, and
(ii) arithmetic and geometric progressions in $N_2$ and $N_3$.

Finally, using the above sets associated to the totient function, we generate an infinite class of subsets of
$\mathbb{N}$, each with asymptotic density zero and containing arbitrarily long arithmetic progressions.

\end{abstract}

\keywords{Euler's function; Sparsely totient numbers; Banach density}


\section{Introduction}
Euler's totient function $\phi(n)$, which enumerates the number of positive integers which are co-prime to 
and less than or equal to $n$, is a classical arithmetical function. It is a well known fact that the
number of solutions to the equation $\phi(x) = m$ is finite for each $m \in \mathbb{N}$ ($\mathbb{N}$ is the set of positive integers).
It is natural, then, to ask the following questions: 
\begin{enumerate}[label=(\roman*)]
 \item For a given $m \in \mathbb{N}$, what is the largest integer $n$ such that $\phi(n) \leq m$?
 \item What are the largest and the smallest integers satisfying $\phi(x) = m$?
 \end{enumerate}
We denote the set $\{x : \phi(x) = m\}$ by $\phi^{-1}(m)$ and the image of $\phi$ by $V$, i.e.
$V = \{\phi(m) : m \in \mathbb{N}\}$. The elements of $V$ are called \textbf{totients}.
For $m \in V$, we define the following quantities with the above questions in view:
\begin{align*}
 N_1(m) &= \max\{x : \phi(x) \leq m\} \\
 N_2(m) &= \max \left(\phi^{-1}(m)\right) \\
 N_3(m) &= \min \left(\phi^{-1}(m)\right) \\
 N_i &= \{N_i(m) : m \in V\} \ \text{ for } i = 1,2,3.
\end{align*}
Note that $N_2(m), N_3(m)$ are defined only on $V$ whereas $N_1(m)$ can be defined on the whole of $\mathbb{N}$.
But this doesn't contribute any new elements to the image $N_1$ of $N_1(m)$, since $N_1(m) = N_1(m - 1)$ if $m \notin V$. Hence,
from here on, we study $N_1(m)$ only for $m \in V$. In 1986, Masser and Shiu \cite{masser} studied many properties of 
$N_1$ and called its elements as `sparsely totient numbers'. They gave the following criteria to find examples
of sparsely totient numbers.
\begin{proposition}[Masser-Shiu, \cite{masser}]\label{sparse_totient_element}
Let $(p_i)_{i=1}^{\infty}$ be the enumeration of the primes in ascending order. Suppose $k\geq 2$, $d\geq 1$, $l\geq 0$ satisfy
  conditions $d < p_{k+1}-1$ and
  $d (p_{k+l}-1) < (d+1)(p_{k}-1).$
Then $dp_1\cdots p_{k-1}p_{k+l}$ is a sparsely totient number.
\end{proposition}
 They also found some nice patterns among sparsely totient numbers.
 \begin{proposition}[Masser-Shiu, \cite{masser}] \label{sparse_totient_property}
 For $n\in N_1$, let $n'$ represent the smallest sparsely totient number greater than $n$. Then
  \begin{enumerate}[label=(\roman*)]
   \item \label{sparse_totient_property1}  $\frac{n'}{n}\rightarrow 1  \text{ as }  n\rightarrow \infty$ and $n\in N_1$.
   \item \label{sparse_totient_property2}For a given prime $p$, $\exists  \ m(p)\in \mathbb{N}$ such that $m)\equiv 0\pmod{p}$
  for all $m\geq m(p)$. 
  \end{enumerate}
\end{proposition}
This proposition suggests that the distribution of  elements of $N_1$ may be very sparse. To study the notion of sparseness of a
subset of integers, we use properties like asymptotic density or Banach density. Asymptotic density gives the fraction of the number of elements of a set in $\mathbb{N}$ whereas
Banach density gives an idea about how locally sparse or dense a set is. For example, the set
$\cup_{n \in \mathbb{N}}[10^n, 10^n + n]$ has asymptotic density zero but it has, in fact,
maximum Banach density of 1. The notion of Banach density will be defined in Section \ref{section2}. The first theorem in this paper measures the densities of sets $N_1,N_2,N_3$ etc.
\begin{theorem}\label{t1}\hspace{0.3cm}
 \begin{enumerate} [ label=(\roman*)]
\item The Banach density of  $N_1$ is zero. 
\item If $f\colon V \rightarrow \mathbb{N}$ is such that $f(m)\in \phi^{-1}(m)$, then the asymptotic density of $f(V)$ is $0$. 
In particular, the asymptotic density of $N_2$ and $N_3$ is zero.
\end{enumerate}
\end{theorem}
More generally, we also look at the Banach density of sets that are images of injective-increasing functions on $\mathbb{N}$.

\begin{theorem}\label{injective_increasing-Banach}
 Let $A \subset \mathbb{N}$. Suppose $f\colon A\rightarrow \mathbb{N}$ is an injective and  increasing function. 
 \begin{enumerate}[label=(\alph*)]
  \item \label{banach_th1a} If the function $\displaystyle\frac{f(n)}{n}$ is increasing on $A$ and $\displaystyle\lim_{\substack{n\rightarrow \infty \\ n \in A}}\displaystyle\frac{f(n)}{n} = \infty$, then the Banach density of $f(A)$ is zero.
  \item \label{banach_th2a} For $A=\mathbb{N}$, if there exists $n_0\in \mathbb{N}$ and positive absolute constants $c_1$ and $c_2$ such that $c_1n\leq f(n)\leq c_2n$ for $n\geq n_0$, then  the Banach density of $f(\mathbb{N})$ is positive.
 \end{enumerate}
\end{theorem}

In \ref{banach_th1a} above, the hypothesis `increasing' for $\frac{f(n)}{n}$ is only a sufficient condition. For instance, if $ BN_1 = \{m \in V : N_1(m) = N_2(m)\},$
then the function $h \colon BN_1 \rightarrow N_1$ given by $h(m)= N_1(m)$ doesn't satisfy this condition but nevertheless, the  Banach density of $N_1$ is zero .

In Section \ref{section3}, we observe that $N_2\supset N_1$ and $N_3\supset \mathbb{P}\setminus \{2\}$ where $\mathbb{P}$
denotes the set of primes. Therefore we look for infinite families of elements in $N_2\setminus N_1$ 
and an infinite family of composite numbers in $N_3$. This leads to our next theorem:

For $r, r_1, r_2 \in \mathbb{N}$ and a prime $q \equiv 3 \pmod{4}$, define
\begin{align*}
R(r_1, r_2) &:= 2.3^{r_1}.5^{r_2}, \ K_{q, r} := 2q^{r+1}, \ k_{q,r} := \begin{cases} q^{r}(q-1)+1 &\mbox{if } q^{r}(q-1)+1 \in \mathbb{P}\\
  q^{r+1} &\mbox{otherwise. } \end{cases}
\end{align*}

 A prime of the form $2^{2^l}+1$ with $l\in \mathbb{N}\cup \{0\}$ is called a Fermat prime. We denote the $j$th Fermat prime by 
 $F_j$. The only known Fermat primes are $$ F_1=3, \ F_2=5, \ F_3=17, \  F_4=257 \text{ and } F_5=65537. $$ Existence of $F_6$
 is not known.

\begin{theorem} \label{t3}
$\mathcal{K}_{\max}, \mathcal{R}$ and $\mathcal{F}$ are infinite subsets of $N_2$ in which only finitely many elements are in $N_1$. 
Moreover, $\mathcal{K}_{\min}$ is an infinite subset of $N_3$ in which infinitely many elements are composite. Here,
\begin{align*}
  \mathcal{K}_{\max}&=\{K_{q,r}\colon q\equiv 3\pmod{4} , r\in\mathbb{N}\},  \\
  \mathcal{K}_{\min}&=\{k_{q,r}\colon q\equiv 3\pmod{4} , r\in\mathbb{N}\},\\
  \mathcal{R}&=\{R(r_1,r_2)\colon r_1, r_2\in\mathbb{N}, r_2>2\},\\
  \mathcal{F}&=\left\{2^a\prod_{i=1}^{k} F_i\colon k\in H;\ a\leq \log_2(F_{k+1}-1) \text{ if $F_{k+1}$  exists and } a\in\mathbb{N} \text{ otherwise} \right\},
\end{align*}
  where $F_j$ denotes the $j$th Fermat prime and $H=\left\{k \in \mathbb{N}\colon F_k \text{ exists}\right\}.$ 
\end{theorem}

\noindent
  From this theorem, we observe that $N_2$ contains infinitely many elements divisible by powers of $n$ where $n$ is $2$, $5$ or a prime $q$ with  $q \equiv 3\pmod{4}$. 
  The infinite family $\mathcal{F}$ of elements of $N_2$ shows the importance of Fermat primes to generate many elements in $N_2$. The family $\mathcal{K}_{\min}$ of elements of $N_3$ 
  shows that $N_3$ contains infinitely many elements divisible by powers of some prime $q$ where  $q\equiv 3\pmod{4}$.

Though we have given examples of infinite families in $N_2$ and $N_3$, there may still be other elements in these sets.
So, we give bounds for general $N_2(m)$ and $N_3(m)$ in the case when $m \not\equiv 0 \pmod{8}$.
We also study properties of the ratio $\frac{N_2(m)}{N_3(m)}$ and geometric progressions contained inside $N_2$ and $N_3$.
The ratios  $\frac{N_2(m)}{N_3(m)}$ are important in the sense that the statement ``$\frac{N_2(m)}{N_3(m)}>1$ for each $m\in \mathbb{V}$'' is
equivalent to  Carmichael's conjecture which asserts that
$|\phi^{-1}(m)|\geq 2 ~\forall ~ m\in V$. 
\begin{theorem}\label{t4}
 Let $m \in V$.
\begin{enumerate} [label= (\roman*)] 
 \item \label{t41}If $m\equiv 2\pmod{4}$ or $4\pmod{8}$, then $m<N_3(m)<2m$ and $2m<N_2(m)<4m.$
 \item \label{t42}There exist infinitely many $m$ such that  $\frac{N_2(m)}{N_3(m)}=2$. Further, if $m\equiv2 \pmod{4}$, then 
  $2 \leq \frac{N_2(m)}{N_3(m)}\leq 3$.
  \item \label{t43}$N_2$ and $N_3$ contain an infinite geometric progression.
\end{enumerate}
\end{theorem}

In Section \ref{section 4}, we discuss about the existence of arithmetic progressions in infinite subsets of natural numbers. 
The famous Szemer\'{e}di's Theorem\cite{gowers} gives a sufficient condition for the existence of arbitrarily long 
arithmetic progressions in a subset of the integers, namely, a positive asymptotic density.
But this is no necessary condition. Therefore we give a class of subsets of the integers having zero asymptotic density
and containing arbitrarily long arithmetic progressions. These sets are formed by taking exactly one element from
each pre-image $\phi^{-1}(m)$, $m \in V$. Theorem \ref{t5} below follows as a consequence by using results 
due to Green-Tao\cite{green-tao} and Erd\H{o}s\cite[Theorem 4]{erdos}. Therefore
\begin{theorem}\label{t5}
 If $f\colon V \rightarrow \mathbb{N}$ is such that $f(m)\in \phi^{-1}(m)$, then $f(V)$ contains arbitrarily long arithmetic
progressions. 
\end{theorem}
Indeed, we observe that these sets satisfy the hypothesis of the so-called
Erd\H{o}s-Tur\'{a}n conjecture\cite[page 4]{gowers} which asserts that if a set $X$ of positive integers 
such that the sum of reciprocals of elements of $X$
diverges, then $X$ contains arbitrarily long arithmetic progressions.

Finally, in Section \ref{section 5}, we pose some questions about elements of $N_2$ and 
Banach density of $N_2$ and $N_3$ arising from the present
work.

We use the following notation in this paper. Let $\mathbb{N}, \mathbb{P}$, $\mathbb{R}^+$ and $\mathbb{Z}$ denote, respectively, the 
set of positive integers, the set of prime numbers, the set of positive real numbers and the set of integers. $p, q$ will always represent prime numbers
unless otherwise mentioned. We write $f(x)=o(g(x))$ if $\frac{f(x)}{g(x)} \rightarrow 0$ as $x \rightarrow \infty$. 
$\lfloor x \rfloor$ denotes the greatest integer less than or equal to $x$, $[a, b]$ denotes the set 
$\{x \in \mathbb{N}: a \leq x \leq b\}$ and similarly for the sets 
$(a, b], [a, b)$ and $(a, b)$ and finally $W(x)$ denotes the set of prime divisors of $x$. By convention, we assume empty products and 
empty sums to take the values 1 and 0 respectively. By ``a divergent sequence $(x_n)$'', we mean that $x_n\rightarrow \infty$ as $n\rightarrow \infty$.

\section{\textbf{Sparse subsets of natural numbers and sparsely totient numbers}}\label{section2}
It is well-known that the set of totients $V$ is sparsely distributed, i.e., has asymptotic density zero (see, for example, \cite{forddensity} and the references therein).
\begin{proposition}[Kevin Ford \cite{forddensity}] \label{Totient_estimates}
 If $V(x)$ is the number of totients less than or equal to $x$, then $$V(x)=\displaystyle\frac{x}{\log x}\exp\{(C+o(1))(\log\log\log x)^2\},$$ where
 $0.81<C<0.82.$ 
\end{proposition}

Here, we study the sparseness of the set of totients $V$, the set of sparsely totient numbers $N_1$ and other
subsets of natural numbers using a generalized version of asymptotic density called Banach density. We will define Banach density using F\o{}lner sequences.
\begin{definition}[F\o{}lner sequence]
 A F\o{}lner sequence in a countable commutative semigroup $(G,+)$ is a sequence $(F_{n})_{n\in\mathbb{N}}$ of
finite subsets of G such that $\forall~ g\in G$,
$$\displaystyle\lim_{n\rightarrow \infty}\displaystyle\frac{|(g+F_{n})\cap F_n|}{|F_n|}=1.$$

\end{definition}
\begin{example}
 In the semigroup $(\mathbb{N},+)$, let $F_n=[\alpha_n, \beta_n]$, with $\beta_n - \alpha_n \rightarrow \infty$ 
 as $n \rightarrow \infty$,  then $(F_n)_{n\in \mathbb{N}}$ is a F\o{}lner sequence.
\end{example}
 
\begin{definition}[Density of a Subset of $\mathbb{N}$]
 Let $(F_{n})_{n\in\mathbb{N}}$ be a F\o{}lner sequence in $\mathbb{N}$ and $A\subset \mathbb{N}$. Then the upper density of $A$ with respect to the F\o{}lner
 sequence $(F_n)_{n\in\mathbb{N}}$ is defined by $$ \overline d_{F_n}(A)=\limsup_{n\rightarrow \infty}\displaystyle\frac{|F_n\cap A|}{|F_n|} $$
 and the lower density of $A$ with respect to the F\o{}lner sequence $(F_n)_{n\in\mathbb{N}}$ is defined by $$ \underline d_{F_n}(A)=\liminf_{n\rightarrow \infty}\displaystyle\frac{|F_n\cap A|}{|F_n|} .$$
 If the upper density and the lower density are equal, then we say that the density of $A$ with respect to the 
 F\o{}lner sequence exists and it equals $$ d_{F_n}(A)=\lim_{n\rightarrow \infty}\displaystyle\frac{|F_n\cap A|}{|F_n|}. $$
 \end{definition}
 \begin{definition}[Asymptotic density]\label{asympdensity}
  The density with respect to the F\o{}lner sequence $([1,n])_{n\in \mathbb{N}}$ is called Asymptotic density. In this case, the upper asymptotic 
  density, the lower asymptotic density  and the density of a subset $A$ are denoted by $\overline d(A)$, $\underline d(A)$
and $d(A)$ respectively.
\end{definition}

\begin{definition}[Banach density]
The Banach density $\displaystyle d^{\ast}(A)$ of $A\subset \mathbb{N}$ is defined by 
$$d^{\ast}(A)=\displaystyle \sup \{\overline d_{F_n}(A)\colon (F_{n})_{n\in\mathbb{N}} \text{ is a F\o{}lner sequence in } \mathbb{N}\}.$$
\end{definition}

\begin{example}
  Banach density of the set of primes is zero. (See \cite[p. 194]{ruzsa}) 
 \end{example}

Using $F_n=[1,n] ~\forall n\in \mathbb{N}$ in the  following proposition, one can observe that the  Banach density of a subset of $\mathbb{N}$
is equal to density of that subset with respect to the F\o{}lner sequence $([t_n+1, t_n+n])_{n\in\mathbb{N}}$ for 
some sequence $(t_n)_{n\in\mathbb{N}}$ in $\mathbb{N}$. Therefore it is enough to consider F\o{}lner sequences formed by intervals in $\mathbb{N}$  
to evaluate Banach density. 
\begin{proposition}[Beiglb\"{o}ck et al., p. 418, \cite{amenable}]
 Given a subset $A$ of $\mathbb{N}$ and any F\o{}lner sequence $(F_{n})_{n\in \mathbb{N}}$, there is a sequence $(t_n)_{n\in\mathbb{N}}$ such that
 $$ d^{\ast}(A)=d_{(F_n+t_n)}(A).$$
\end{proposition}

\subsection{Proof of Theorem \ref{t1}}

We now evaluate the densities of the sets $N_1,N_2,N_3$ and more generally, for sets of the form $f(V)$ where $f:V\rightarrow \mathbb{N}$ is an injective map such that $f(m)\in \phi^{-1}(m)$.  
For this, we start with some necessary lemmas. 

\begin{lemma}\label{Totient_lemma4}
 Suppose that $(F_n)_{n\in\mathbb{N}}$ is a F\o{}lner sequence on $\mathbb{N}$ defined by $$F_n=(x_n,x_n(1+\alpha_n)],$$  where
$(x_n)_{n=1}^{\infty}$ is a sequence in $\mathbb{N}$  and $(\alpha_n)_{n=1}^{\infty} $ is a sequence of positive reals such that $\alpha_n>\alpha_0>0$ for each $n \in \mathbb{N}$. Then
$\overline d_{F_n}(V)=0$.
\end{lemma}
\begin{proof}
Since $|F_n|=x_n\alpha_n\rightarrow \infty$ as $n\rightarrow \infty$, we can choose $n_0\in \mathbb{N}$ such that 
 $\log\log\log (1 + \alpha_n)x_n>0$  $\forall ~ n\geq n_0$. For $n \geq n_0$, we get
\begin{align*}
 \displaystyle\frac{|V\cap F_n|}{|F_n|}&=\displaystyle\frac{|V\cap [1,x_n(1+\alpha_n)]|-|V\cap [1,x_n]|}{|F_n|}
 \leq \displaystyle\frac{|V\cap [1,x_n(1+\alpha_n)]|}{|F_n|}.
 \end{align*}
  Using the estimate of $V(x)$ from Proposition \ref{Totient_estimates} (with the same constant $C$ appearing there), we get
\begin{align*}
 \frac{|V\cap F_n|}{|F_n|} & \leq \displaystyle\frac{(1+\alpha_n)}{\alpha_n \log((1+\alpha_n) x_n)}\left(\exp((C+o(1))(\log\log\log (1+\alpha_n)x_n)^2)\right).
 \end{align*}
 Since $\alpha_n > \alpha_0$ for each $n \in \mathbb{N}$, applying the inequality $y^2 < e^y$ for $y >0$ gives us 
 \begin{align*}
 \frac{|V\cap F_n|}{|F_n|} &\leq\displaystyle\frac{(1+\alpha_0)}{\alpha_0}\left( \displaystyle\frac{\left(\log((1+\alpha_n) x_n)\right)^{C+o(1)}}{\log((1+\alpha_n) x_n)}\right)\rightarrow 0 \text{ as } n\rightarrow \infty,
\end{align*}
as $C<1$. Hence $\overline d_{F_n}(V)=0$.
\end{proof}

\begin{lemma} \label{Totient_lemma5}
 Suppose $A,B\subset \mathbb{N}$ and $g\colon A\rightarrow B$ is an injective map satisfying $g(x)\leq x~\forall ~x\in A$. Let $(F_n)_{n\in \mathbb{N}}$ 
 be a F\o{}lner sequence in $\mathbb{N}$ such that  $F_n=(a_n, x_n]$ and $(a_n)_{n=1}^{\infty}$ is bounded. Then $\overline d_{F_n}(B)=0$ implies $\overline d_{F_n}(A)=0$.
\end{lemma}
\begin{proof}

Since $g\colon A\rightarrow B$ is an injective map and $g(x)\leq x~\forall ~x\in A$, we have $g: F_n\cap A\rightarrow g(A)\cap [1,x_n]$ is injective for each $n\in \mathbb{N}$.
It follows that $|F_n\cap A|\leq |g(A)\cap [1,x_n]| \ \forall \ n\in \mathbb{N}$.  Since $ g(A)\cap [1,x_n]\subset \left(g(A)\cap F_n\right)\cup [1,a_n]$, we get that $|F_n\cap A|\leq |g(A)\cap F_n|+ |[1,a_n]|$. Therefore
 $$\frac{|F_n\cap A|}{|F_n|}\leq \frac{|F_n\cap g(A)|}{|F_n|}+\frac{a_n}{|F_n|}\leq \frac{|F_n\cap g(A)|}{|F_n|}+\frac{a}{|F_n|},$$
 where $a$ is an upper bound of the sequence $(a_n)_{n=1}^{\infty}$. Since $g(A) \subset B$ and $\overline d_{F_n}(B)=0$, we conclude that
$\overline d_{F_n}(A)=0.$
\end{proof}
\begin{corollary} \label{N_1_2_3_asymp_density}
 Let $(F_n)_{n\in\mathbb{N}}$ be a F\o{}lner sequence on $\mathbb{N}$ defined by $F_n=(a_n,x_n],$ where $(a_n)_{n=1}^{\infty}$ is bounded. 
 If $f\colon V \rightarrow \mathbb{N}$ is such that $f(m)\in \phi^{-1}(m)$, then $\overline d_{F_n}(f(V))=0$. In particular, the asymptotic density of $N_1,N_2$  and $N_3$ is zero.
\end{corollary}
\begin{proof}
 Consider $g \colon f(V)\rightarrow V$ defined by $g(n) = \phi(n)$. This is an injective map satisfying $g(x)\leq x$ $\forall \ x\in f(V)$.
 Since $\overline d_{F_n}(V)=0$, by Lemma  \ref{Totient_lemma4}, it follows that $\overline d_{F_n}(f(V))=0$ by applying Lemma \ref{Totient_lemma5}. In particular,
 $\overline d_{F_n}(N_i)=0$ for $i= 2, 3$. Also, $N_1\subset N_2$ so that  $\overline d_{F_n}(N_1)=0$.
\end{proof} 

\begin{proposition}\label{Banach_sparse_totient_number}
 Banach density of $N_1$ is zero.
\end{proposition}
\begin{proof}
  Let $(F_n)_{n\in\mathbb{N}}$ be a F\o{}lner sequence on $\mathbb{N}$ defined by $F_n=(x_n,x_n+y_n],$ where $(x_n)_{n=1}^{\infty}$ and $(y_n)_{n=1}^{\infty}$ are  sequences
  in $\mathbb{N}$ with $y_n\rightarrow \infty$. To show that the Banach density of $N_1$ is zero, it is enough to prove that
  $\overline d_{F_n}(N_1)=0$ for the two cases when 
  (i) $(x_n)_{n=1}^{\infty}$ is bounded, and (ii) $(x_n)_{n=1}^{\infty}$ is divergent.

  If $(x_n)_{n=1}^{\infty}$ is a bounded sequence, then Corollary \ref{N_1_2_3_asymp_density} establishes that $\overline d_{F_n}(N_1)=0$. So we can
  assume that $(x_n)_{n=1}^{\infty}$ is a divergent  sequence. If $p$ is a prime number, then Proposition \ref{sparse_totient_property}\ref{sparse_totient_property2} gives
  the existence of an  element $m_0(p)\in \mathbb{N}$ such that
  $N_1(m)\equiv 0 \pmod{p}$ for each $m\geq m_0(p)$. Since $(x_n)_{n=1}^{\infty}$ is divergent, we can choose $n_0(p)\in \mathbb{N}$ such that $x_n>N_1(m_0(p))$ for 
  each $n>n_0(p)$. Then $F_n$ contains at most $ \frac{|F_n|}{p}+1$ elements of $N_1$ for $n>n_0(p)$. So, given a prime $p$, there exists
  $n_0(p)\in \mathbb{N}$ such that
  $$ n>n_0(p)\Rightarrow \displaystyle\frac{|F_n\cap N_1|}{|F_n|}\leq\displaystyle \frac{1}{p} + \frac{1}{y_n} .$$ 
  Since $y_n \rightarrow \infty$ as $n \rightarrow \infty$, this means that $\overline d_{F_n}(N_1) \leq \frac{1}{p}$. 
  As this holds for each prime $p$, we conclude that $\overline d_{F_n}(N_1)=0$ if $(x_n)_{n=1}^{\infty}$ is divergent. Therefore, $d^*(N_1)=0$.

  \end{proof}

Hence, proof of Theorem  \ref{t1} is complete by collecting Propositions \ref{Banach_sparse_totient_number} and Corollary \ref{N_1_2_3_asymp_density}.  

\subsection{Some criteria for sparse sets in $\mathbb{N}$}
We have studied the Banach densities of specific sets like $V$ and $N_1$. Now, we are going to investigate the behavior of 
sparse sets which are the images of injective, increasing functions on $\mathbb{N}$. We proceed to the proof of Theorem $\ref{injective_increasing-Banach}$.
 \begin{proof}[Proof of  Theorem \ref{injective_increasing-Banach}\ref{banach_th1a}]
  Let (a,b] be an interval in $\mathbb{N}$ such that $a,b\in f(A)$. Since $f$ is injective and increasing, it follows that
  $\{x\in A\colon f(x)\in (a,b]\}= \{x\in A\colon x\in (f^{-1}(a),f^{-1}(b)]\}$. \\
  Therefore, \begin{equation*}\displaystyle \frac{|(a,b]\cap f(A)|}{b-a} \leq\displaystyle \frac{f^{-1}(b)-f^{-1}(a)}{b-a}= \displaystyle\frac{1}{b-a}\left(\frac{b}{m_b}-\frac{a}{m_a}\right), \label{fA}\end{equation*}
  where $m_a=\displaystyle\frac{a}{f^{-1}(a)}$ and  $m_b=\displaystyle\frac{b}{f^{-1}(b)}$. Since the function
  $\displaystyle\frac{f(n)}{n}$ is increasing on $A$, we get $m_a\leq m_b$. It follows that 
   $$\displaystyle \frac{|(a,b]\cap f(A)|}{b-a}\leq \displaystyle\frac{1}{m_b}.$$ \label{eq1}
   Suppose that $((a_n,b_n])_{n\in \mathbb{N}}$ is a F\o{}lner sequence with $|(a_n,b_n]\cap A|>1$. 
   Then for each $n\in\mathbb{N}$, there
   exist $a_n',b_n' \in f(A)$ such that $(a_n,b_n]\cap f(A)=[a_n',b_n']\cap f(A)$ and $a_n <a_n'< b_n'\leq b_n$. 
   So, $$\limsup_{b_n-a_n\rightarrow \infty}\displaystyle \frac{|(a_n,b_n]\cap f(A)|}{b_n-a_n}\leq \limsup_{b_n-a_n\rightarrow \infty}\displaystyle \frac{|(a_n',b_n']\cap f(A)|}{b_n'-a_n'}\leq\limsup_{b_n-a_n\rightarrow \infty} \displaystyle\frac{1}{m_{b_n'}}.$$
   Note that $b_n-a_n\rightarrow \infty \Rightarrow b_n \rightarrow \infty.$ 
   We claim that $b_n ' \rightarrow \infty$. If not, there exists a subsequence $\{b_{n_k}'\}$ of $\{b_n'\}$ such that
   $b_{n_k}' \leq l \ \forall \ k \in \mathbb{N}$. By the definition of $b_n'$, we know that $(b_{n_k}', b_{n_k}] \cap f(A)
   = \varnothing$ for each $k$. In particular, $(l, b_{n_k}] \cap f(A)= \varnothing$ for all $k$. Since $b_{n_k} \rightarrow \infty$
   as $k \rightarrow \infty$, this implies that $f(x) \leq l \ \forall \ x \in A$. But this is a contradiction since
   $f$ is strictly increasing and hence grows indefinitely. Thus, $b_n ' \rightarrow \infty$. Therefore,
   $$\limsup_{b_n-a_n\rightarrow \infty}\displaystyle \frac{|(a_n,b_n]\cap f(A)|}{b_n-a_n}\leq \limsup_{b_n' \rightarrow \infty} \displaystyle\frac{1}{m_{b_n^{'}}}=0,$$
since $(m_{b_n'})_{n=1}^{\infty}$ is a subsequence of the divergent sequence $\left(\frac{f(n)}{n}\right)_{n\in A}$. 

Hence, $\overline d_{F_n}(f(A))=0$ for all F\o{}lner
sequences $(F_n)_{n\in \mathbb{N}}$ with $F_n=(a_n, b_n]$. So, $d^{\ast}(f(A))=0.$
\end{proof}
\begin{proof}[Proof of Theorem \ref{injective_increasing-Banach}\ref{banach_th2a}]
 Choose $r\in \mathbb{N}\setminus \{1\}$ such that $c_2<(r-1)c_1$ and consider the F\o{}lner sequence 
  $(F_n)_{n\in\mathbb{N}}$ given by $F_n=(n,rn]$.
 Since $f(\mathbb{N})$ is an infinite set, there exists $k'\in \mathbb{N}$ such that for each $n\geq k'$, one can choose integers $a_n, b_n \in f(\mathbb{N})$ such that $(a_n,b_n)\cap f(\mathbb{N})=(n,rn] \cap f(\mathbb{N})$ and $a_n\leq n< rn\leq b_n$.
  Since $f$ is injective and increasing, it follows that
  $\{x\in \mathbb{N}\colon f(x)\in (a_n,b_n)\}= \{x\in \mathbb{N}\colon x\in (f^{-1}(a_n),f^{-1}(b_n))\}$. We now have
 \begin{align}
  \displaystyle \frac{|F_n\cap f(\mathbb{N}) |}{|F_n|}& = \displaystyle \frac{f^{-1}(b_n)-f^{-1}(a_n)-1}{(r-1)n}. \label{positivebanacheq1}
 \end{align}
  From the hypothesis, we have
 \begin{equation}
  c_1x\leq f(x)\leq c_2 x \ \forall \ x\geq n_0. \label{positivebanacheq2}
 \end{equation}

 Choose positive integer $k\geq k'$ such that $f^{-1}(a_n),f^{-1}(b_n)\geq n_0$ for all $n\geq k$. Inserting the values of $f^{-1}(a_n), f^{-1}(b_n)$ obtained from inequality \eqref{positivebanacheq2} in the equation \eqref{positivebanacheq1} for $n\geq k$, we get
 \begin{align*}
  \displaystyle \frac{|F_n\cap f(\mathbb{N}) |}{|F_n|}&\geq \displaystyle\frac{1}{(r-1)n}\left(\displaystyle\frac{b_n}{c_2}-\displaystyle\frac{a_n}{c_1}\right)-\frac{1}{(r-1)n}\\
 &> \displaystyle\frac{1}{(r-1)n}\left(\displaystyle\frac{rn}{(r-1)c_1}-\displaystyle\frac{n}{c_1}\right)-\frac{1}{(r-1)n} \ \ (\text{since } c_2<(r-1)c_1)\\
 &=\displaystyle\frac{1}{c_1(r-1)^2}-\frac{1}{(r-1)n}.
 \end{align*}
 Therefore, $$\limsup_{n \rightarrow \infty} \frac{|F_n\cap f(\mathbb{N}) |}{|F_n|} \geq \frac{1}{c_1(r-1)^2}.$$
 Hence, $\overline d_{F_n}(f(\mathbb{N}))>0$ and therefore $d^{*}(f(\mathbb{N}))>0$. 
 \end{proof}
 
 We drop the `increasing function' hypothesis on $\frac{f(n)}{n}$ in Theorem \ref{injective_increasing-Banach}\ref{banach_th1a} and through the two examples given below
show that the conclusion on Banach density may or may not hold.
 \begin{example}
Given $k,l\geq 2$, we define a map $f_{k,l}\colon \mathbb{N}\rightarrow\mathbb{N}$ by
\begin{align*}
f_{k,l}(x)= \begin{cases} k^{2nl}+(l-1)x &\mbox{if } k^{2n}\leq x< k^{2n+1},\\
x^{l} &\mbox{if } k^{2n+1}\leq x<k^{2n+2}. \end{cases}
\end{align*}
One can see that $f_{k, l}$ is injective and increasing. The function $\frac{f_{k,l}(n)}{n}$ is divergent but not increasing.
Note that $f_{k, l}\left([k^{2n}, k^{2n + 1})\right)$ is an arithmetic progression of length $k^{2n + 1} - k^{2n}$ and common difference
$l-1$. Therefore, $f_{k,l}(\mathbb{N})$ contains arbitrarily long arithmetic progressions with common difference $l-1$. Hence, it has positive Banach density.
\end{example}
The above example suggests that the `increasing' hypothesis on $\frac{f(n)}{n}$ is necessary. However, this is not always the case
as the next example shows.

Suppose $BN_1=\{m \in V \colon ~N_1(m)=N_2(m)\}$. Then the function $h\colon BN_1\rightarrow N_1$ defined by $h(m)=N_1(m)$, is both bijective and increasing. 
By a result of Sanna \cite[Lemma 2.1]{sanna} on the asymptotic of $N_1(m)$, it is readily seen that $\frac{h(m)}{m}=\frac{N_1(m)}{m} \rightarrow \infty$ as $m \rightarrow \infty.$
\begin{lemma}[Sanna\cite{sanna}]\label{N1m_bound}
$N_1(m)\sim e^{\gamma}$ $m\log\log m$ as $m\rightarrow \infty$, where $\gamma$ is the Euler-Mascheroni constant.
\end{lemma}

The next proposition tells us that $\frac{h(m)}{m}$ is not an increasing function.
\begin{proposition}
$\displaystyle\frac{h(n)}{n}\colon BN_1\rightarrow N_1$ is not an increasing function.
\end{proposition}
\begin{proof}
 For  $p\in \mathbb{P}$, define $$X_{p}=\prod_{q\in\mathbb{P},q\leq p}q.$$ 
 
 Let $p_1$ and $p_2$ be  two consecutive  primes such that $3<p_1<p_2$. Let $a=X_{p_1}$ , $b=\frac{X_{p_2}}{p_1}$, $M_{a}=\phi(a)$ and $M_{b}=\phi(b)$. 
 Then by Proposition \ref{sparse_totient_element} and the definition of $h$, we get
 $$h(M_a)=N_1(M_a)=a  \text{ and }h(M_b)=N_1(M_b)=b.$$
 The proof of the equation $h(M_b)=b$ uses a result of Nagura which states that $(n,1.2n)\cap \mathbb{P}\neq \varnothing$ $\forall$ $n>25.$
 This gives us 
 $$\displaystyle\frac{h(M_a)}{M_a}=\displaystyle\prod_{q\in \mathbb{P}, q\leq p_1}\frac{q}{q-1}=\left(\frac{p_1}{p_1-1}\right)\left(\frac{p_2-1}{p_2}\right) \frac{h(M_b)}{M_b}.$$ 
 Since $p_1<p_2$,  it follows that
 $$\displaystyle\frac{h(M_a)}{M_a}>\frac{h(M_b)}{M_b} \ \text{ but } \ M_{a}=\frac{p_1-1}{p_2-1}M_b<M_b.$$
 Therefore $\frac{h(n)}{n}$ is not an increasing function.
\end{proof}

In contrast to $f_{k, l}$, though $\frac{h(n)}{n}$ is not increasing, we know that $d^{\ast}(h(BN_1))=d^{\ast}(N_1)=0$.
Therefore we observe that if we remove the condition of `increasing map'
on $\frac{f(n)}{n}$, then both the possibilities, namely, Banach density is zero or positive may occur.


\section{\textbf{Explicit construction of elements of $N_2$ and $N_3$}}\label{section3}
\subsection{Proof of Theorem \ref{t3}}
Now we move on to the study of $N_2$ and $N_3$. 
As we know that $N_1 \varsubsetneq N_2$, we give explicit examples of infinite families of elements in $N_2 \setminus N_1$.
Since $\phi(p)=p-1$ and $\phi(p-1)<p-1$ for an odd prime $p$, this implies that $N_3(p-1)=p$. So, $\mathbb{P}\setminus\{2\} \subset N_3.$
We are going to show that infinitely many composite numbers also lie in $N_3$. First, we give the following useful definitions:

\begin{definition}[$k_{q,r}$ and $K_{q,r}$]
 For $q \in \mathbb{P}$ and $r, r_1, r_2 \in \mathbb{N}$, define 
 \begin{align*}
k_{q,r} &\colon = \begin{cases} q^{r}(q-1)+1 &\mbox{if } q^{r}(q-1)+1 \text{ is a prime},\\
  q^{r+1} &\mbox{otherwise }. \end{cases} \\
  K_{q,r} &\colon= 2q^{r+1}, \hspace{0.75cm} R(r_1,r_2) = 2\cdot 3^{r_1}\cdot 5^{r_2}. 
  \end{align*}
\end{definition}

The following lemma gives a description of the elements of $\phi^{-1}(m)$ for $m \equiv 2 \pmod{4}.$ This will be useful to 
construct families of elements in $N_2$ and $N_3$ as indicated above.

\begin{lemma}\label{klee2mod4}
Let $A(m)$ denote the number of solutions to the equation $\phi(x)=m.$
Suppose $m>2$ and $m \equiv 2 \pmod{4}$. Then,
\begin{enumerate}  [ label=(\roman*)]           
\item \label{klee1} Every element of $\phi^{-1}(m)$ is of the form $p^{\alpha}$ or $2p^{\alpha}$ where $p \equiv 3 \pmod{4}.$ 
\item \label{klee2} $A(m) = 0, 2$ or $4$.
\item \label{klee3} If $A(m) = 2$, then $\phi^{-1}(m) = \{p^{\alpha}, 2p^{\alpha}\}$ for some $p \equiv 3 \pmod{4}$, $\alpha \geq 1$ and
 if $A(m) = 4$, then $\phi^{-1}(m) = \{p^{\beta}, 2p^{\beta}, q, 2q\}$ for some $p, q \equiv 3 \pmod{4}, \text{ with } p < q, \ \beta > 1.$
\end{enumerate}
\end{lemma}
\begin{proof}
For the proof of \ref{klee1} and \ref{klee2}, see \cite{klee}. 
(iii) If $p \in \mathbb{P},\ p \equiv 3 \pmod{4}$, then $\phi(p^{\alpha})=\phi(2p^{\alpha})$ for $\alpha \geq 0.$ If $A(m)=2$, 
then from \ref{klee1}, $\phi^{-1}(m)= \{p^{\alpha}, 2p^{\alpha}\}$ for some prime $p \equiv 3 \pmod{4}.$ On the other hand, if $A(m)=4$,
then $\phi^{-1}(m) = \{p^{\beta}, 2p^{\beta}, q^{\gamma}, 2q^{\gamma}\}$ for some $p, q \equiv 3 \pmod{4}$ and $\beta, \gamma \geq 1.$ Now, $p^{\beta}\neq q^{\gamma}$ and $\phi(p^{\beta})=\phi(q^{\gamma}) \Rightarrow p \neq q$. Without loss of
generality, let us assume that $p < q$. This means that $\beta > \gamma.$ Now, $\phi(p^{\beta})=\phi(q^{\gamma}) \Rightarrow p^{\beta - 1}(p-1)=q^{\gamma-1}(q-1)$.
If $\gamma > 1$, then it means that $q \mid (p-1),$ a contradiction to $p < q$. Thus, $\gamma=1.$

Therefore, $\phi^{-1}(m) = \{p^{\beta}, 2p^{\beta}, q, 2q\}$ for some $p, q \equiv 3 \pmod{4}, \text{ with } p < q, \ \beta > 1$ in the case of $A(m)=4.$
\end{proof}
\begin{lemma}\label{N2N3points_small_lemma}
 Let $q$ be a prime greater than $7$. Then there exists a unique odd integer $n\in \{q+2, q+4\}$ such that $n\equiv 0\pmod{3}$, $\gcd(n,q)=1$
 and $\phi(n)< q$.
\end{lemma}
\begin{proof}
 Since $q$ is a prime and $q>7$, one can choose the unique integer $n\in \{q+2, q+4\}$ such that $n\equiv 0\pmod{3}$. Since $q$ is odd,
 $\gcd(q,n)=1$.  Let $n=3^rl$ with $r,l\in \mathbb{N}$ and $3 \not\mid l$. Hence $\phi(n)=2\times 3^{r-1}\phi(l)\leq 2\times 3^{r-1}l=\frac{2n}{3}\leq\frac{2q+8}{3}<q$ if $q\geq 11$. 
\end{proof}

\begin{proposition}   \label{N2N3points}
 Suppose that $r \in \mathbb{N}$ and $q$ is a prime satisfying $q\equiv 3 \pmod{4}$. Then 
 \begin{enumerate} [ label=(\roman*)]
 \item $N_2(q^r(q-1))=K_{q,r}$ and $N_3(q^r(q-1))= k_{q,r}$.
 \item $K_{q,r} \notin N_1$ except when $(q,r)= (3,1)$.
\end{enumerate}
\end{proposition}

\begin{proof}
 Let $m=q^{r}(q-1)$. Then $q\equiv 3\pmod{4}\Rightarrow m\equiv 2 \pmod{4}$. Since $\phi(q^{r+1})=q^{r}(q-1)$, it follows
 that $\phi^{-1}(m)$ is non-empty. Hence, applying Lemma \ref{klee2mod4} we get $\phi^{-1}(m)=\{q_1^{\alpha}, 2q_1^{\alpha}\}$
 or $\{q_2^{\beta}, 2q_2^{\beta}, q_3, 2q_3\}$, where $q_2<q_3$, $m>2, \ \alpha \geq 1$
 and $\beta >1$.

  If $\phi^{-1}(m)=\{q_1^{\alpha}, 2q_1^{\alpha}\}$, then $q_1=q$ and $\alpha=r+1$ as $\phi(q^{r+1})=q^{r}(q-1)= \phi(2q^{r+1})$.
  Hence $N_2(q^r(q-1))=K_{q,r}$ and $N_3(q^r(q-1))= q^{r+1}$ in this case. Since $\phi^{-1}(m)$ doesn't contain primes,
  this means that $m+1=q^r(q-1)+1$ is composite and hence $k_{q, r}=q^{r+1}=N_3(m).$ 

  If suppose $\phi^{-1}(m)=\{q_2^{\beta}, 2q_2^{\beta}, q_3, 2q_3\}$, where $q_2, q_3 \equiv 3 \pmod{4}, q_2<q_3$ and $\beta >1$. 
  If $q^{r}(q-1)+1 \text{ is a prime}$, then $q^{r}(q-1)+1 \in \phi^{-1}(m)$. It follows that $N_3(q^r(q-1))= k_{q,r}$ 
  in this case. Since the only prime in $\phi^{-1}(m)$ is $q_3$, we get $q_3=q^{r}(q-1)+1=N_3(q^r(q-1))$. This means that
  $q_2^{\beta}>q_3$. Now, note that $q^{r+1}> q^r(q-1) + 1$ and $q^{r+1}$ is the only odd composite number in $\phi^{-1}(m)$. 
  Thus, $q_{2}^{\beta}=q^{r+1}$, i.e., $q_2=q$ and $\beta=r + 1$. Therefore, $N_2(q^r(q-1))=2q^{r + 1}$.
  On the other hand, if $q^{r}(q-1)+1 \text{ is not a prime}$, then no element of $\phi^{-1}(m)$ can be prime. 
  But this contradicts the fact that $q_3\in \phi^{-1}(m)$. Therefore,
  $$N_2(q^r(q-1))=K_{q,r} \text{ and }N_3(q^r(q-1))= k_{q,r}.$$

 Coming to the proof of $(ii)$, If $q> 7$, then Lemma \ref{N2N3points_small_lemma} gives the odd integer $n\in \{q+2, q+4\}$ such that $n\equiv 0\pmod{3}$, $\gcd(n,q)=1$
 and $\phi(n)< q$. Now we observe that $2nq^{r-1}>2q^{r}$ but $\phi(2nq^{r-1})\leq \phi(2q^{r})$. Hence $2q^r \not \in N_1 \ \forall ~ q>7$.
 If $q=7$, then $2\times 3^2\times7^{r-1}>2\times 7^r$ but $\phi(2\times 3^2\times7^{r-1})\leq\phi(2\times 7^r)$ and hence $2\times 7^r \not \in N_1 \ \forall ~ r\geq 1$.
  If $q=5$, then $12\times 5^{r-1}>2\times 5^r$ but $\phi(12\times 5^{r-1})\leq\phi(2\times 5^r)$. Hence $2\times 5^r \not \in N_1 \ \forall ~ r\geq 1$.
 Since $\phi(20\times 3^{r-2})<\phi(2\times 3^r)$ but $20\times 3^{r-2}>2\times 3^r$ for $r\geq3$, it follows that $2\times 3^r\not \in N_1 ~\forall ~r\geq 3. $  
\end{proof}

From Proposition \ref{N2N3points}, we see that $K_{q, r} \in N_2\setminus N_1 \ \forall \ r \geq 3$. So, for each $q \equiv 3 \pmod{4}$,
this gives an infinite family of elements in $N_2 \setminus N_1$. But the proposition does not ensure the presence of infinitely many composite
numbers in $N_3$. For this, we require that $k_{q, r}$ is composite for infinitely many $(q, r)$. If $q=3$, note that $2.3^r + 1$
is divisible by $5$ when $r \equiv 3 \pmod{4}$. In other words, $k_{3, r}$ is composite for infinitely many $r$. 
So, from Proposition \ref{N2N3points}, we see that $N_3$ contains infinitely many composite numbers.

Now, we give another infinite family of elements in $N_2\setminus N_1$. First, we state some definitions, four preliminary lemmas and then prove the
two main lemmas which together construct an infinite two-parameter family of elements in $N_2\setminus N_1$.
\begin{definition}[D(A,B)]
Let $A$ and $B$ be two finite subsets of $\mathbb{P}$. Then $D(A,B)$ is defined by
\begin{align*}
 D(A,B):=\left(\prod_{q\in A}\frac{q-1}{q}\right)\left(\prod_{q\in B}\frac{q}{q-1}\right).
\end{align*}
\end{definition}

\begin{lemma} \label{Xnpmainlemma2}
 Suppose that $y,x\in \mathbb{N}\setminus \{1\}$. If $\phi(y)\leq \phi(x)$ and $y>x$, then 
 $D(W(y),W(x))<1$. 
\end{lemma}
\begin{proof}
We know that  $$\phi(y)=y\displaystyle\prod_{q\in W(y)}\left(\displaystyle\frac{q-1}{q}\right) \text{ and } \phi(x)=x\displaystyle \prod_{q\in W(x)}\left(\displaystyle\frac{q-1}{q}\right).$$ 
This gives us
  $$1\geq\displaystyle\frac{\phi(y)}{\phi(x)}=\displaystyle \frac{y D(W(y),W(x))}{x}>D(W(y),W(x)),$$
  since $\phi(y)\leq \phi(x)$ and $y>x$.
  \end{proof}
  
\begin{lemma} \label{Xnplemma_D(A,B)1}
Let $A$ and $B$ be two finite subsets of $\mathbb{P}$ such that $|B|\leq|A|$. 
If $\min(B\setminus A)>\max(A)$ or $B\subset A$, then $D(B,A)\geq1$.
\end{lemma}
\begin{proof}
 If $B \subset A$, then $D(B, A) = \left(\prod_{q \in A\setminus B}\frac{q}{q-1}\right) \geq 1$.  
In the case when $B \not\subset A$, define an injective map $f\colon B\rightarrow A$ such that $f(x)=x$ for $x\in A\cap B$. 
If $\min(B\setminus A)>\max(A)$, it follows that
 $f(x)\leq x~\forall ~x\in B$. Therefore, $$D(B,A)\geq \prod_{x\in B}\left(\frac{(x-1)f(x)}{x(f(x)-1))}\right)=\prod_{x\in B}\left(\frac{xf(x)-f(x)}{xf(x)-x}\right)\geq 1.$$
 \end{proof}
 
\begin{lemma} \label{Rr1r2lemma1}
 Let $a,k\in \mathbb{N}\setminus \{1\}$ and $k\leq a$. Suppose that $ x_1, x_2, \dots , x_k$ are non-negative integers such that at least two of them are postive. Then
 $$\sum_{i=1}^{k}a^{x_i}\leq a^{x_1+x_2+\dots+x_k}.$$
\end{lemma}
\begin{proof}
 Since atleast two of the $k$ integers $ x_1, x_2, \dots , x_k$ are positive, it follows that $a^{x_i}\leq a^{x_1+x_2+\dots+x_k-1}$ $\forall$ $i\in [1,k]$. Therefore 
$$\sum_{i=1}^{k}a^{x_i}\leq ka^{x_1+x_2+\dots+x_k-1}\leq a^{x_1+x_2+\dots+x_k},$$
since $k \leq a$.
 \end{proof}
 
\begin{lemma} \label{Rr1r2lemma2}
Let $x,y,k\in\mathbb{N}$ such that  $x,y,k \geq 2$ and $k\leq\min\{x,y\}$.  Suppose that for each $i\leq k$, $a_i$ and $b_i$ are non-negative integers such that $a_i+b_i\neq 0$. If $a_1+a_2+\dots +a_k=t$ and $b_1+b_2+\cdots +b_k=u$, then
$$  \sum_{i=1}^{k}x^{a_i}y^{b_i}\leq x^ty^u.$$
\end{lemma}
\begin{proof}
Given $k \geq 2$. If possible, suppose that both the sequences $(a_i)_{i=1}^{k}$, $(b_i)_{i=1}^{k}$ contain at most one positive integer.
If $k \geq 3$, then there exists $j \in [1, k]$ such that $a_j + b_j = 0$, a contradiction. Hence, $k=2$ and exactly one of
$\{a_1, a_2\}$ and one of $\{b_1, b_2\}$ are positive with $a_i + b_i \neq 0$ for $i \in [1, 2]$. Therefore, we can assume $a_1, b_2>0$
without loss of generality. We need to show that $$x^{a_1} + y^{b_2} \leq x^{a_1}y^{b_2}$$ in this case. Since $x, y \geq 2$ and
$a_1, b_2 \in \mathbb{N}$, it is enough to show that $v+w \leq vw$ for $v,w \in \mathbb{N}\setminus \{1\}$. This happens iff
$v \leq w(v-1)$ iff $v/(v-1) \leq w$ which is true since the left hand side is not greater than 2.

On the other hand, if one of the sequences, say $(a_i)$, has at least two positive elements, then by Lemma \ref{Rr1r2lemma1},
we have $$ \sum_{i=1}^{k}x^{a_i}y^{b_i} \leq \left(\sum_{i=1}^{k}x^{a_i}\right)y^u \leq x^{t}y^{u},$$ since $k\leq \min\{x,y\} \leq x$.
\end{proof}
\begin{definition}[Valuation]
 Let $p$ be a prime number. Then the $p$-valuation on the integers $\mathbb{Z}$ is the map
 $v_{p} \colon \mathbb{Z}\rightarrow \mathbb{N}\cup \{0, \infty\}$ defined
 by $v_{p}(0)=\infty$ and $v_{p}(n)=r$ for $n \neq 0$, where $r$ is the largest non-negative integer such that $p^{r} \mid n.$
\end{definition}
\begin{lemma} \label{Rr1r2lemmaMain1}
Suppose $r_1, r_2, y \in \mathbb{N}$ satisfy $\phi(y)=\phi(R(r_1,r_2)),\ |W(y)|=4$, $v_2(y)=1$, $v_3(y)=0$
and $v_5(y)=0$. Then $y\leq R(r_1,r_2)$. 
\end{lemma}
\begin{proof}
 Since $|W(y)|=4$, $v_2(y)=1$, $v_3(y)=0$
and $v_5(y)=0$, we can write $y=2\left(q_1^{v_{q_1}(y)}q_2^{v_{q_2}(y)}q_3^{v_{q_3}(y)}\right)$ where $q_1,q_2,q_3$ are distinct primes greater than $6$
 and $v_q(y)\geq 1$ for $q\in\{q_1,q_2,q_3\}$. Since
 $\phi(y)=\phi(R(r_1,r_2))$, it follows that $v_{q_1}(y)=v_{q_2}(y)=v_{q_3}(y)=1$ and hence 
 $$\left(\frac{q_1-1}{2}\right)\left(\frac{q_2-1}{2}\right)\left(\frac{q_3-1}{2}\right)=5^{r_2-1}3^{r_1-1}.$$
 Therefore, for each $i\in\{1,2,3\}$, we can write
 $q_i=2\cdot3^{a_i}5^{b_i}+1$ such that 
$$a_1+a_2+a_3=r_1-1 , \ b_1+b_2+b_3=r_2-1,$$ $$a_1+b_1\neq 0, \ a_2+b_2\neq0, \ a_3+b_3\neq0,$$ $$a_1,a_2,a_3,b_1,b_2,b_3\geq 0.$$
\begin{align}
\therefore y&=2\left(2\cdot 3^{a_1}5^{b_1}+1\right)\left(2\cdot 3^{a_2}5^{b_2}+1\right)\left(2\cdot 3^{a_3}5^{b_3}+1\right)\nonumber\\
&=2\left(2^33^{r_1-1}5^{r_2-1}+2^2\left( \sum_{i=1}^{3}3^{r_1-a_i-1}5^{r_2-b_i-1}\right)+2\left( \sum_{i=1}^{3}3^{a_i}5^{b_i}\right)+1\right)\nonumber\\
&=2\left(2^33^{r_1-1}5^{r_2-1}+2^23^{r_1-1}5^{r_2-1}\left( \sum_{i=1}^{3}3^{-a_i}5^{-b_i}\right)+2\left( \sum_{i=1}^{3}3^{a_i}5^{b_i}\right)+1\right) \label{Rr1r2eq3}.
\end{align}
Since $a_i+b_i\geq 1$ for each $i=1,2,3$, we have 
\begin{equation}
\sum_{i=1}^{3}3^{-a_i}5^{-b_i}\leq \sum_{i=1}^{3}3^{-(a_i+b_i)}\leq 1. \label{Rr1r2eq4}
\end{equation}
Applying Lemma \ref{Rr1r2lemma2}, we have
\begin{equation}
  \sum_{i=1}^{3}3^{a_i}5^{b_i}\leq 3^{r_1-1}5^{r_2-1} \label{Rr1r2eq5}.
\end{equation}
Inserting the inequalities \eqref{Rr1r2eq4} and \eqref{Rr1r2eq5} into the right side of \eqref{Rr1r2eq3}, we get
\begin{align*}
y&\leq2\left(2^33^{r_1-1}5^{r_2-1}+2^23^{r_1-1}5^{r_2-1}+2\cdot3^{r_1-1}5^{r_2-1}+1\right) \\
&\leq2\cdot 3^{r_1-1}5^{r_2-1}(8+4+2+1)= R(r_1,r_2).
\end{align*}
\end{proof}
\begin{lemma}
\label{Rr1r2lemmaMain2}
Suppose $r_1, r_2, y\in \mathbb{N}$ and $r_2>2$ satisfy $\phi(y)=\phi(R(r_1,r_2))$, $|W(y)|=4$, $v_2(y)=1$, $v_3(y)\geq 1$
and $v_5(y)=0$. Then $y\leq R(r_1,r_2)$.
\end{lemma}
\begin{proof}
 Since $|W(y)|=4$, $v_2(y)=1$, $v_3(y)\geq 1$
and $v_5(y)=0$, we can write $y=2\left(3^{v_3(y)}q_1^{v_{q_1}(y)}q_2^{v_{q_2}(y)}\right)$ where $q_1,q_2$ are distinct primes greater that $6$ and $v_q(y)\geq 1$ for $q\in\{q_1,q_2,3\}$. Since
 $\phi(y)=\phi(R(r_1,r_2))$, it follows that $v_{q_1}(y)=v_{q_2}(y)=1, v_3(y)\leq r_1$ and hence 
 $$\left(\frac{q_1-1}{2}\right)\left(\frac{q_2-1}{2}\right)=5^{r_2-1}3^{r_1-v_3(y)}.$$
 Therefore we can write
 $q_1=2\cdot 3^{a_1}5^{a_2}+1$ and $q_2=2\cdot 3^{b_1}5^{b_2}+1$ such that $a_1+b_1=r_1-v_3(y)$, $a_2+b_2=r_2-1,\ a_1+a_2\neq0,\ b_1
 +b_2\neq 0$ and $a_1,a_2,b_1,b_2\geq 0$. 
 \begin{align}
  \therefore y&=2\cdot 3^{v_3(y)}\left(2\cdot 3^{a_1}5^{a_2}+1\right)\left(2\cdot 3^{b_1}5^{b_2}+1\right) \nonumber\\
  &=2\left(2^2\cdot 3^{a_1+b_1+v_3(y)}5^{a_2+b_2}+2\left(3^{a_1+v_3(y)}5^{a_2}+ 3^{b_1+v_3(y)}5^{b_2}\right)+3^{v_3(y)}\right).\nonumber
 \end{align}
 Inserting  $a_1+b_1=r_1-v_3(y)$ and  $a_2+b_2=r_2-1$  in the above equation, we get
 \begin{equation}
 y = 2\left(2^2\cdot 3^{r_1}5^{r_2-1}+2\left(3^{r_1-b_1}5^{a_2}+ 3^{r_1-a_1}5^{b_2}\right)+3^{v_3(y)}\right). \label{Rr1r2eq1}
 \end{equation}
 
 From Lemma \ref{Rr1r2lemma1}, we have 
 \begin{equation}5^{a_2}+5^{b_2}\leq\begin{cases}
                    5^{r_2-1} &\mbox{if } a_2,b_2>0,\\
                    1+5^{r_2-1} &\mbox{else. } 
                   \end{cases} \label{Rr1r2eq2}
\end{equation}
We are now going to consider the following cases depending on the value of $a_1$ and $b_1$.

\textbf{\underline{Case 1}:} If $a_1$ and $b_1$ are positive, then $a_1+b_1\geq 2$. Using this along with the conditions $a_1+b_1+v_3(y)=r_1$ and $v_3(y)\geq 1$, we get $r_1\geq 3$ and $v_3(y)\leq r_1-2$.
Applying these in the right side of equation \eqref{Rr1r2eq1}, we have 
\begin{align*}
 y&\leq 2\left(2^2\cdot 3^{r_1}5^{r_2-1}+2\cdot3^{r_1-1}\left(5^{a_2}+ 5^{b_2}\right)+3^{r_1-2}\right). 
\end{align*}
Inserting the value of $5^{a_2}+5^{b_2}$ from \eqref{Rr1r2eq2} in the above inequality, we have
\begin{align*}
y&\leq 2\cdot 3^{r_1-2}\left(36\cdot 5^{r_2-1}+6\left(1+5^{r_2-1}\right)+1\right)\\ 
&\leq 2\cdot 3^{r_1-2}\left(45\cdot 5^{r_2-1}-3\cdot 5^{r_2-1}+7\right)\\
&\leq 2\cdot 3^{r_1}5^{r_2}=R\left(r_1,r_2\right) \ \text{  for each } r_2\geq 2.
\end{align*}

\textbf{\underline{Case 2}:} If  $a_1=0$ and $b_1=0$, then $v_3\left(y\right)=r_1$ due to the fact that 
$a_1+b_1=r_1-v_3\left(y\right)$. Applying these in equation \eqref{Rr1r2eq1}, we have
$$y\leq 2\cdot 3^{r_1}\left(4\cdot 5^{r_2-1}+2\left(5^{a_2}+ 5^{b_2}\right)+1\right). $$
Since $a_1+a_2\neq0$ and $b_1+b_2\neq 0$, it follows that $a_2,b_2>0$. Hence $a_2,b_2\leq r_2-2$ because $a_2+b_2=r_2-1$. 
Using this in the previous inequality gives us
\begin{align*}
 y&\leq 2\cdot 3^{r_1}\left(4\cdot 5^{r_2-1}+4\cdot 5^{r_2-2}+1\right)\\
 &\leq 2\cdot 3^{r_1}\left(5^{r_2}-5^{r_2-2}+1\right)\leq R\left(r_1,r_2\right), 
\end{align*}
since $r_2=1+a_2+b_2\geq 3$.

\textbf{\underline{Case 3}:} The remaining cases are in which exactly one of $a_1$ and $b_1$ is zero. Without loss of generality, assume that $a_1=0$ and $b_1\neq 0$.

If $b_2\geq 1$, we get $r_2\geq 2$ and $a_2\leq r_2-2$, because $a_2+b_2=r_2-1$. Since $a_1=0$ and $a_1+a_2\neq0$,
we have $a_2\geq 1$ and hence $b_2\leq r_2-2$. Equation \eqref{Rr1r2eq1} above gives us
\begin{align*}
y&\leq 2\left(2^2\cdot 3^{r_1}5^{r_2-1}+2\left(3^{r_1-1}5^{r_2-2}+ 3^{r_1}5^{r_2-2}\right)+3^{r_1}\right)\\
 &\leq 2\cdot 3^{r_1}\left(2^2\cdot 5^{r_2-1}+ 4 \cdot 5^{r_2-2}+1\right) \\
 &\leq 2\cdot 3^{r_1}\left(2^2\cdot 5^{r_2-1}+ 5^{r_2-1}-5^{r_2-2}+1\right)\\
 &\leq 2\cdot 3^{r_1}\left(5^{r_2}-5^{r_2-2}+1\right)\\
 &\leq 2\cdot 3^{r_1}5^{r_2}, \text{ since } r_2\geq 2. 
 \end{align*}

 Now, in the case $b_2=0$, we have $a_2=r_2-1$. Since $a_1=0$ and $a_2+a_1\neq 0$, it follows that $a_2\geq 1$ and hence $r_2\geq2$. Then
\eqref{Rr1r2eq1} gives us
 \begin{align*}
y&\leq 2\left(2^2\cdot 3^{r_1}5^{r_2-1}+2\left(3^{r_1-1}5^{r_2-1}+ 3^{r_1}\right)+3^{r_1}\right)\\
&\leq 2\left(3^{r_1}5^{r_2}-3^{r_1-1}5^{r_2-1}+3^{r_1+1}\right) \leq R\left(r_1,r_2\right),\text{ since } r_2 > 2.
 \end{align*}

Hence $y\leq R(r_1,r_2)$ for $r_1, r_2 \in \mathbb{N},\ r_2>2.$
\end{proof}

\begin{proposition}\label{RinN2}
$R(r_1,r_2)$ lies in $N_2$ for each $r_1, r_2 \in \mathbb{N},\ r_2>2.$
\end{proposition}
\begin{proof}
 Let $y$ be an even number such that $\phi(y)=\phi(R(r_1,r_2))$. Since $v_2(\phi(R(r_1,r_2)))=3$, it follows that 
 $v_2(\phi(y))=3$. This means that $y$ can have atmost $4$ prime factors. If $|W(y)|\leq 3$, then $$D(W(y), W(R(r_1,r_2)))\geq 1 $$
 by Lemma \ref{Xnplemma_D(A,B)1}. This gives $y\leq R(r_1,r_2)$ by Lemma \ref{Xnpmainlemma2}. Now, we consider the case
 $|W(y)|=4$. Since $v_2(\phi(y))=3$, it follows in this case that $v_2(y)=1$.
  
  Suppose $v_5(y)\geq 1$, then $v_2(\phi(y))\geq 4$. It follows that $v_2(\phi(R(r_1,r_2)))\geq 4$ which contradicts the fact that
   $v_2(\phi(R(r_1,r_2)))=3$. Therefore $v_5(y)=0$.
   
   If $v_3(y)\geq 1$, then Lemma \ref{Rr1r2lemmaMain2} ensures that $y\leq R(r_1,r_2)$. If $v_3(y)=0$, then Lemma \ref{Rr1r2lemmaMain1} gives $y\leq R(r_1,r_2)$. 
   Therefore $R(r_1,r_2)\in N_2$ in any case.
\end{proof}

\begin{remark}\label{N2question}
From Proposition \ref{sparse_totient_property}\ref{sparse_totient_property2}, we get that any element in $N_1$, all of whose prime factors are less than some prime $p$, has bounded exponents
for its prime factors. But, as seen above from Proposition \ref{RinN2}, this is not the case for elements in $N_2$. Infact,  
$R(r_1, r_2) \in N_2$ for $r_1, r_2 \in \mathbb{N},\ r_2>2.$

So, this raises the following question:
For a given odd prime $p$, do there exist non-negative integers $d_q$ corresponding to each odd prime $q < p$ such that  
$$r_q>d_q \text{ for each } q<p \Rightarrow 2\prod_{2<q < p} q^{r_q} \in N_2?$$
The numbers $R(r_1, r_2)$ and $K_{3,r}$ answer this question in the affirmative for $p=7$ and $p=5$ respectively.
\end{remark}

Now we are going to give another infinite family of elements in $N_2$ in which the odd prime factors of elements are
Fermat primes.
\begin{lemma}\label{N2_fermat_lemma_one}
 If $\phi(x)=2^r$ for some $x,r\in\mathbb{N}$, then there exist $b,n\in \mathbb{N}\cup \{0\}$ and a sequence of distinct Fermat primes 
 $(F_{i_j})^{j=n}_{j=1}$ such that
 $$x=2^b\prod_{j=1}^{n}F_{i_j}.$$
\end{lemma}
\begin{proof}
 We observe that if $\phi(x)=2^r$ and if an odd prime $q \mid x$, then $(q-1) \mid 2^r$ which implies that $q$ is of the form $2^l + 1$ for some $l \in \mathbb{N}$.
But it is well-known that if $2^l + 1$ is a prime, then $l=2^{\alpha}$ for some $\alpha \geq 0$ (see \cite[Theorem 17]{hardy}).
Hence, $q=2^{2^{\alpha}}+1$, a Fermat prime. Also, $v_{q}(x)=1$  for each such $q \mid x$. If not, then $q \mid \phi(x)=2^r$, 
a contradiction since $q$ is odd. Therefore $x$ will be of the form
 $x=2^b\prod_{j=1}^{n}F_{i_j}$   where  $i_j\in \mathbb{N}, \ b, n\in \mathbb{N}\cup\{0\}.$

\end{proof}

\begin{proposition}\label{N2_fermat}
Let $F_j$ denote the $j$th Fermat prime for $j \in \mathbb{N}$. Suppose $F_1$, $F_2$, $\dots$, $F_k$ exist. 
If $F_{k + 1}$ also exists, then $2^aF\in N_2$ where $F=\prod_{i=1}^{k}F_i$ and $1\leq a \leq \log_2(F_{k+1}-1)$. 
If $F_{k+1}$ does not exist, then $2^aF\in N_2$ for each $a\in\mathbb{N}.$
\end{proposition}

\begin{proof}
Define $y:=2^a\prod_{i=1}^{k}F_i$ with $a\in \mathbb{N}$.  To prove $y\in N_2$, it is enough to show that if $x$ is  any even integer satisfying $\phi(y)=\phi(x)$
then $x\leq y$. This can be observed using the fact that elements of $N_2$ are even.

Let $x$ be an even integer satisfying $\phi(x)=\phi(y)$. Since $\phi(y)= 2^r$ for some $r\in \mathbb{N}$, it follows that
$\phi(x)=2^r$ for some $r\in \mathbb{N}$. Then  Lemma \ref{N2_fermat_lemma_one} gives 
$x=2^b\prod_{j=1}^{n}F_{i_j}$ for some $b \in \mathbb{N}$ and $n\in \mathbb{N}\cup\{0\}$. Since $\phi(x)=\phi(y)$, 
we have
    \begin{equation}a=b+\sum_{j=1}^{n}\log_2(F_{i_j}-1)-\sum_{j=1}^{k}\log_2(F_j-1).\label{N2_fermat_eq1}\end{equation}
 If $F_{k+1}$ exists, then 
 \begin{align*}
|W(x)|> |W(y)|&\Rightarrow \sum_{j=1}^{n}\log_2(F_{i_j}-1)-\sum_{j=1}^{k}\log_2(F_j-1)\geq \log_2(F_{k+1}-1)\\
  &\Rightarrow a\geq b+ \log_2(F_{k+1}-1) \text{ by equation } \eqref{N2_fermat_eq1} \\
  &\Rightarrow a\geq 1+ \log_2(F_{k+1}-1) \text{ since } b\in\mathbb{N}.
 \end{align*}
 Therefore \begin{align*}
a\leq \log_2(F_{k+1}-1)&\Rightarrow |W(x)|\leq |W(y)|\\%
&\Rightarrow D(W(x), W(y))\geq 1 \text{ using Lemma } \ref{Xnplemma_D(A,B)1}\\
&\Rightarrow  x\leq y \text{ by Lemma } \ref{Xnpmainlemma2}.
           \end{align*}
On the other hand, if $F_{k+1}$ doesn't exist, then
 $W(x)\subset W(y)$ for each $a\in\mathbb{N}$. It follows that  $D(W(x), W(y))\geq 1$  using Lemma \ref{Xnplemma_D(A,B)1}. Then Lemma \ref{Xnpmainlemma2} gives  $x\leq y$. 

\end{proof}

Only five Fermat primes are known to date. From the above proposition, one can see that there exist elements in $N_2\setminus N_1$ which
are divisible by arbitrarily large powers of $2$. In all the earlier results, the elements obtained were divisible by 2 but not by 4.
\begin{corollary} 
For a positive integer $r$, there exist infinitely many integers $l$ such that $l \equiv 0 \pmod{2^r}$ and $l \in N_2\setminus N_1.$
\end{corollary}

\begin{definition}
Let $F_j$ denote the $j$th Fermat prime for $j \in \mathbb{N}$ and let $H=\left\{k \in \mathbb{N}\colon F_k \text{ exists}\right\}.$ 
Define
 \begin{align*}
  \mathcal{K}_{\max}&=\{K_{q,r}\colon q\equiv 3\pmod{4} , r\in\mathbb{N}\},  \
  \mathcal{K}_{\min}=\{k_{q,r}\colon q\equiv 3\pmod{4} , r\in\mathbb{N}\},\\
  \mathcal{R}&=\{R(r_1,r_2)\colon r_1, r_2\in\mathbb{N}, r_2>2\},\\
  \mathcal{F}&=\left\{2^a\prod_{i=1}^{k}F_i \colon k\in H; a\leq \log_2(F_{k+1}-1) \text{ if $F_{k+1}$  exists and } a\in\mathbb{N} \text{ otherwise} \right\}.
 \end{align*}
\end{definition}

By collecting Propositions \ref{N2N3points}, \ref{RinN2} and \ref{N2_fermat}, we get that (i)  $\mathcal{K}_{\max}, \mathcal{R}$ and $\mathcal{F}$ are infinite subsets of $N_2$ and (ii) $\mathcal{K}_{\min}$ is an infinite subset of $N_3$ in which infinitely many elements are composite.
Proposition \ref{sparse_totient_property}\ref{sparse_totient_property2} gives that only finitely many elements of $\mathcal{K}_{\max}, \mathcal{R}$ and $\mathcal{F}$ belong to $N_1$. %
{\renewcommand{\thetheorem}{\ref{t3}}
\begin{theorem}
$\mathcal{K}_{\max}, \mathcal{R}$ and $\mathcal{F}$ are infinite subsets of $N_2$ in which only finitely many elements are in $N_1$. 
$\mathcal{K}_{\min}$ is an infinite subset of $N_3$ in which infinitely many elements are composite.
\end{theorem}
}
\subsection{Proof of Theorem \ref{t4}}
In the previous results, we looked at several families of elements in $N_2$ and $N_3$.  
Now, we would like to compare the values of $N_2(m)$ and $N_3(m)$. In the following proposition, we are going to give upper and lower bounds for $N_2(m)$ and $N_3(m)$ and also look at the ratio $N_2(m)/N_3(m)$. 
\begin{lemma} \label{N2N3ratio_(2)}

 Let $m$ be an odd integer. If $u$ is an odd integer satisfying $\phi(u)=4m$, then 
 $$u=\displaystyle \frac{(2z_1+1)(2z_2+1)}{z_1z_2}m,~\text{ or }~ \displaystyle \frac{(4z_3+1)}{z_3}m,$$
 where $z_1z_2\mid m$, $z_3\mid m$, $z_1\neq z_2$ and $2z_1+1, 2z_2+1, 4z_3+1$ are primes. Also $4m<u\leq 7m$.
\end{lemma}

\begin{proof}
 Any odd integer $u$ satisfying $\phi(u)=4m$ can have at most two prime factors. If $u$ has two distinct prime factors $q_1$ and $q_2$,
such that $q_1<q_2$ and $q_1, q_2 \equiv 3 \pmod{4}$ (since $v_2(\phi(u))=2$), then
 $$u\left(\frac{(q_1-1)(q_2-1)}{q_1q_2}\right)=4m, \text{ i.e., } u=\frac{4q_1q_2}{(q_1-1)(q_2-1)}m .$$
Since $u,q_1,q_2$ and $m$ are all odd, we have $q_1=2z_1+1, \ q_2=2z_2+1$ for some odd integers $z_1, z_2$, with \ $z_1 < z_2$. 
Moreover, $q_1q_2\mid u \Rightarrow (q_1-1)(q_2-1) \mid \phi(u)=4m$, i.e., $z_1z_2 \mid m$. Using this in the value of $u$, we have
 $$u=\displaystyle \frac{(2z_1+1)(2z_2+1)}{z_1z_2}m,$$ where $z_1z_2$ divides $m$.
Clearly $u>4m$ and it can take a maximum value of $7m$ if $z_1 = 1, z_2 = 3$. Therefore $u\leq 7m$.
 If $u$ has only one prime factor $q$ with $q \equiv 5 \pmod{8}$ (since $v_2(\phi(u))=2$), 
 then $$u=\frac{4qm}{q-1}.$$ Now, $q \equiv 5 \pmod{8}$, so $q=4z_3+1$ for some odd integer 
 $z_3$. Therefore
 $$u=\frac{(4z_3+1)}{z_3}m.$$
 Note that $z_3,4z_3 + 1$ are co-prime integers and hence $z_3 \mid m$. Clearly $u>4m$ and 
 it can take a maximum value of $5m$ if $z_3 = 1.$
\end{proof}

Now we proceed to prove Theorem \ref{t4}\ref{t41} and \ref{t4}\ref{t42}.
\begin{proof}[Proof of Theorem \ref{t4}\ref{t41}]
 If $m\equiv 2\pmod{4}$ and $\phi(x)=m$ is solvable, then by Lemma \ref{klee2mod4}, $m=q^{r}(q-1)$ for some $q\equiv 3 \pmod{4}$ . 
 Since $q^r(q-1)+1\leq q^{r+1}\leq \frac{3m}{2}$, we have $m<N_3(m)\leq \frac{3m}{2}$ and
 $2m<N_2(m) \leq 3m$ by Proposition \ref{N2N3points}.
 
  We move onto the next case, i.e., $m \equiv 4 \pmod{8}$. Firstly, note that the proposition is true for $m=4$ since $N_3(4)=5$
  and $N_2(4)=12$. So, we can assume that $m \geq 12$. If $\phi(x)=m$ for $m \equiv 4 \pmod{8}$, then
  $v_2(x)\leq 2$.
  
  Suppose that there exists an integer $z$ such that $v_2(z)=2$ and $\phi(z)=m$ where $m=4m_0, \ m_0$ being an 
  odd integer. If $z=4y$, then $y$ is an odd integer satisfying $\phi(x)=2m_0$. Then by Lemma \ref{klee2mod4}, $y=p^{\alpha}$ for some $\alpha \geq 1$ and 
  prime $p \equiv 3\pmod{4}$. Therefore $z=4p^{\alpha}$. If $p>3$, then $3p^{\alpha}<z<6p^{\alpha}$ and $\phi(3p^{\alpha})=\phi(z)= \phi(6p^{\alpha})$. If $p=3$ and $\alpha>1$, then $ 7\times 3^{\alpha-1}<z< 14 \times 3^{\alpha-1}$ and again they have the same $\phi$ value. 
  Hence, if $\phi(z)=m, \ m \equiv 4 \pmod{8}$ and $v_2(z)=2$, then $z \notin \{N_2(m), N_3(m)\}$. Moreover, an odd integer
  $l \in \phi^{-1}(m) \text{ iff } 2l \in \phi^{-1}(m)$. Therefore, $v_2(N_3(m))=0$ and $v_2(N_2(m))=1$.
  Now, note that $N_3(m)$ and $\frac{N_2(m)}{2}$ are odd integers satisfying $\phi(x)=m=4m_1$ where $m_1$ is odd.
  Therefore, by Lemma \ref{N2N3ratio_(2)}, we have $m<N_3(m)$ and $ \frac{N_2(m)}{2} \leq\frac{7m}{4}$ and the result follows.  
\end{proof}
\begin{proof}[Proof of Theorem \ref{t4}\ref{t42}]
  By Lemma \ref{klee2mod4}, if $m \in V$, then $m \equiv 2 \pmod{4} \iff m=q^r(q-1)$ for some prime $q \equiv 3 \pmod{4}, \ r \geq 0$.
  Now, by Proposition \ref{N2N3points},  if $q^r(q-1)+1$ is composite, then 
   
   $$\frac{N_2(q^r(q-1))}{N_3(q^r(q-1))}=2.$$ 
  Else, if $q^r(q-1)+1$ is prime, then $$\frac{N_2(q^r(q-1))}{N_3(q^r(q-1))}=\frac{2q^{r+1}}{q^r(q-1)+1}=\frac{2q}{q-1+\frac{1}{q^r}}.$$

It is readily seen that the rightmost quantity lies between 2 and 3 since $r \geq 0, \ q \geq 3$.
\end{proof}
  
To prove Theorem \ref{t4}\ref{t43}, we give examples of infinite length geometric progressions in $N_2$ and $N_3$. For each prime $q \equiv 3 \pmod{4}$, note that 
$\{K_{q, r}\}_{r \in \mathbb{N}}$ is a geometric progression in $N_2$ with common ratio $q$. Also, we see that $\{R(r_0, r)\}_{r=3}^{\infty}$
is a geometric progression in $N_2$ with common ratio 5.

Now, we turn our attention to geometric progressions in $N_3$.  %
We construct an infinite geometric progression in $N_3$ with the help of the following lemma.

 \begin{lemma}\label{N3infgp}
Let $q$ be a prime satisfying $q \equiv 3 \text{ or } 7 \pmod{20}$. Then the set $\{r \colon q^{r}(q-1) + 1 \text{ is composite}\}$ 
contains an infinite arithmetic progression.
 \end{lemma}
\begin{proof}
Suppose that $q \equiv 3 \pmod{20}$. We observe that $q^{r}(q-1) + 1$ is divisible by $5$ for $r \equiv 3 \pmod{4}$.
Now, if $q \equiv 7 \pmod{20}$, we see that $q^{r}(q-1) + 1$ is divisible by $5$ for $r \equiv 2 \pmod{4}$. So, in any case,
the set $\{r \colon q^{r}(q-1) + 1 \text{ is composite}\}$ contains an infinite arithmetic progression.
\end{proof}  
If $q \equiv 3 \text{ or } 7 \pmod{20}$, then $k_{q, r} \in N_3$ for each $r \in \mathbb{N}$. As the set $S=\{r \colon q^{r}(q-1) + 1 \text{ is composite}\}$ 
contains an infinite arithmetic progression, the set $\{k_{q, r} \colon r \in S\}$ contains an infinite geometric progression.
So corresponding to each such $q$, there is an infinite geometric progression. This implies an infinite family of such 
geometric progressions in $N_3$ due to the following result of Dirichlet\cite[Theorem 15, page 16]{hardy}:

 \begin{proposition}[Dirichlet's Theorem]\label{dirichlet}
 Suppose $(a, q) = 1$. Then there are infinitely many primes $p$ satisfying $p \equiv a \pmod{q}.$
 \end{proposition}
\section{\textbf{Arithmetic progressions in sparse sets}}\label{section 4}
Szemer\'{e}di's Theorem assures the existence of arbitrarily long arithmetic progressions in a set having positive asymptotic 
density. But the converse isn't necessarily true. For example, 
the set of prime numbers has zero asymptotic density but contains arbitrarily long arithmetic progressions, as proved by Green and Tao\cite{green-tao}.
\begin{definition}
Let $A\subset\mathbb{P}$. Then, $Rd(A)=\limsup_{N\rightarrow \infty}\frac{|A\cap [1, N]|}{\pi(N)}$ defines
the relative density of $A$ with respect to $\mathbb{P}$, where $\pi(N)$ denotes the number of primes less than or equal to $N$.
\end{definition}

\begin{proposition}[Green-Tao\cite{green-tao}]\label{green_tao}
Let $A$ be any subset of the prime numbers of positive relative upper density $Rd(A)$.
Then $A$ contains infinitely many arithmetic progressions of length $k$ for all $k$.
\end{proposition}

Let $f\colon V \rightarrow \mathbb{N}$ be such that $f(m)\in \phi^{-1}(m)$. By Corollary \ref{N_1_2_3_asymp_density}, $\overline d(f(V))=0$.
We observe that it satisfies the hypothesis of the following famous conjecture of Erd\H{o}s and Tur\'{a}n:

\begin{conjecture}[Erd\H{o}s-Tur\'{a}n]
 If $A\subset \mathbb{N}$ is such that $\sum_{n\in A}n^{-1}$ diverges, then $A$ contains arbitrarily long arithmetic progressions.
\end{conjecture}

\begin{proposition}
 Let $f\colon V \rightarrow \mathbb{N}$ be such that $f(m)\in \phi^{-1}(m)$. Then $\displaystyle \sum_{m\in V}\frac{1}{f(m)}$ diverges.
\end{proposition}
\begin{proof}
 We have $$\displaystyle \sum_{n\in N_2}\frac{1}{n}\leq\displaystyle \sum_{m\in V}\frac{1}{f(m)}\leq\displaystyle \sum_{n\in N_3}\frac{1}{n}.$$ 
 If $m\equiv 2\pmod{4}$, then  $N_2(m) \leq 3N_3(m)$ by Proposition  \ref{t4}\ref{t42}. So, it follows that 
 $$\displaystyle \sum_{m\in V}\frac{1}{N_2(m)} \geq \displaystyle \sum_{\substack{m\equiv 2 \pmod{4} \\ m \in V}}\frac{1}{3N_3(m)} \geq \displaystyle \sum_{p\equiv 3\pmod{4}}\frac{1}{3p},$$
 since $N_3$ contains each odd prime. The right most series is divergent (see \cite[Chapter 4]{davenport}) and the result follows.
 \end{proof}

 Hence we may expect arbitrarily long arithmetic progressions in $f(V)$. Indeed, the following result due to Erd\H{o}s\cite[p. 15]{erdos} and
 Proposition \ref{green_tao} above, together confirm our intuition.
\begin{proposition}[Erd\H{o}s, \cite{erdos}]\label{erdos_totients}
 Suppose $m \in V$ with $\left|\phi^{-1}(m)\right|=k$ for $k\geq 2$. Then there exists a set $P\subset \mathbb{P}$ such that 
 $Rd(P) > 0$ and for each $p\in P$, $\phi^{-1}(m(p-1))=p\phi^{-1}(m)$. 
 \end{proposition}

\begin{proof}[Proof of Theorem \ref{t5}]
Consider the set $V_1=\{m\in V\colon |\phi^{-1}(m)|=2\}$ and let $m'\in V_1$. Then, by Proposition \ref{erdos_totients},
there exists $P_1\subset \mathbb{P}$ such that $Rd(P_1)>0$ and $m'(p-1)\in V_1$ for each $p\in P_1$.
Now, consider the sets $P_2=\{p\in P_1\colon pN_2(m')\in f(V)\}$ and $P_3=\{p\in P_1\colon pN_3(m')\in f(V)\}$. 
By the definition of the set $f(V)$, we have $P_2\cap P_3=\varnothing$ and $P_2\cup P_3=P_1$. Therefore at least one of the sets
$P_2$ or $P_3$ has positive relative density in $\mathbb{P}$ and thus, by Proposition \ref{green_tao}, 
contains arbitrarily long arithmetic progressions.
Therefore one of the subsets $N_2(m')P_2$ or $N_2(m')P_3$ of $f(V)$ contains arbitrarily long arithmetic progressions and
the result follows. 
\end{proof}
\section{Questions}\label{section 5}
As discussed in Remark \ref{N2question}, we raise the following question about elements in $N_2$.
\begin{question}
 For a given odd prime $p$, do there exist non-negative integers $d_q$ for each prime $q < p$ such that 
$$ r_q>d_q \text{ for each prime} \ q<p \Rightarrow 2\prod_{2<q < p} q^{r_q} \in N_2.$$
\end{question}

In relation to densities of $N_2$ and $N_3$, we ask:
\begin{question}
What is the Banach density of $N_2$ and $N_3$? 
\end{question}
In fact, except for F\o{}lner sequences of the type $(a_n, a_n(1 + \alpha_n)]$, where $\alpha_n \rightarrow 0,\ a_n \rightarrow \infty$,
and $\alpha_n a_n \rightarrow \infty$, one can see that the upper density with respect to other F\o{}lner sequnces is zero.

\end{document}